\documentclass[12pt]{article}

\usepackage{geometry} 
\usepackage{amsthm,amssymb}
\usepackage{amsmath,amsfonts,latexsym}
\usepackage{latexsym}
\usepackage{float}
\usepackage{xcolor}
\usepackage{url}

\def\titlerunning#1{\gdef\titrun{#1}}
\makeatletter
\def\author#1{\gdef\autrun{\def\and{\unskip, }#1}\gdef\@author{#1}}
\def\address#1{{\def\and{\\\hspace*{18pt}}\renewcommand{\thefootnote}{}%
\footnote {#1}}%
\markboth{\autrun}{\titrun}}
\makeatother
\def\email#1{\hspace*{4pt}{\em e-mail}: #1}

\newtheorem{thm}{Theorem}[section]

\newtheorem{cor}[thm]{Corollary}
\theoremstyle{definition}
\newtheorem{rem}[thm]{Remark}
\newtheorem{defi}[thm]{Definition}

\newtheorem{con}[thm]{Conjecture}

\def\GF{{\rm GF}}
\def\Cay{{\rm Cay}}

\begin{document}

\titlerunning{}
\title{Divisible design Cayley digraphs}

\author{Dean Crnkovi\' c, Hadi Kharaghani and Andrea \v Svob}

\maketitle

\address{D. Crnkovi\'{c}, A. \v Svob: Department of Mathematics, University of Rijeka, Radmile Matej\v{c}i\'c 2, 51000 Rijeka, Croatia;
\email{\{deanc,asvob\}@math.uniri.hr}
\and
H. Kharaghani: Department of Mathematics and Computer Science, University of Lethbridge, Lethbridge, Alberta, T1K 3M4, Canada;
\email{kharaghani@uleth.ca} 
}

\begin{abstract}
Divisible design digraphs which can be obtained as Cayley digraphs are studied. A characterization of divisible design Cayley digraphs in terms of the generating sets is given. Further, 
we give several constructions of divisible design Cayley digraphs and classify divisible design Cayley digraphs on $v \le 27$ vertices.
\end{abstract}

\bigskip



\section{Introduction}\label{intro}

We assume that the reader is familiar with the basic facts of group theory,
graph theory and design theory. We refer the reader to \cite{bjl, atlas, diestel, r} on terms not defined in this paper.

\bigskip

A graph $\Gamma$ can be interpreted as a design by taking the vertices of
$\Gamma$ as points, and the neighbourhoods of the vertices as blocks.
Such a design is called a neighbourhood design of $\Gamma$. The adjacency
matrix of $\Gamma$ is the incidence matrix of its neighbourhood design.

A $k$-regular graph on $v$ vertices with the property that any two distinct
vertices have exactly $\lambda$ common neighbours is called a $(v,k, \lambda)$-graph
(see \cite{rudvalis}).
The neighbourhood design of a $(v,k, \lambda)$-graph is a symmetric $(v,k, \lambda)$
design. Haemers, Kharaghani and Meulenberg have defined divisible design graphs
(DDGs for short) as a generalization of $(v,k, \lambda)$-graphs (see \cite{ddg}).

Divisible design digraphs, a directed graph version of divisible design graphs, were introduced in \cite{ddd}.

A directed graph (or digraph) is a pair $D=(V,E)$, where $V$ is a finite
nonempty set of vertices and $A$ is a set of ordered pairs (arcs) $(x,y)$
with $x,y \in V$ and $x \neq y$. A digraph $D$ is asymmetric if 
$(x,y) \in A$ implies $(y,x) \notin A$. If $(x,y)$ is an arc, we will say that $x$ dominates $y$ or that $y$ is dominated by $x$. A digraph $D$ is called regular of degree $k$ if each vertex of 
$\Gamma$ dominates exactly $k$ vertices and is dominated by exactly $k$ vertices.

We call a digraph $D$ on $v$ vertices doubly regular with parameters 
$(v,k, \lambda)$ if it is regular of degree $k$ and, for any distinct 
vertices $x$ and $y$, the number of vertices $z$ that dominates both 
$x$ and $y$ is equal to $\lambda$ and the number of vertices 
$z$ that are dominated by both $x$ and $y$ is equal to $\lambda$.

A digraph $D=(V,A)$ on $v$ vertices $V= \{ x_1, \ldots ,x_v\}$
may be characterized by its adjacency matrix, an 
$v \times v$ $( 0,1 )$-matrix $A=[a_{ij}]$ defined by 

$$a_{ij}=1 \qquad {\rm if \ and\ only \ if} \qquad (x_i,x_j) \in A.$$

The adjacency matrix of a doubly regular digraph is an incidence matrix of a 
symmetric design.

\begin{defi}
Let $D$ be a regular asymmetric digraph of degree $k$ on $v$ vertices.
$\Gamma$ is called a divisible design digraph (DDD for short) 
with parameters $(v,k, \lambda_1, \lambda_2, m,n)$ if the vertex set 
can be partitioned into $m$ classes of size $n$, such that
for any two distinct vertices $x$ and $y$ from the same class, 
the number of vertices $z$ that dominates or being dominated by both 
$x$ and $y$ is equal to $\lambda_1$,
and for any two distinct vertices $x$ and $y$ from different classes, 
the number of vertices $z$ that dominates or being dominated by both 
$x$ and $y$ is equal to $\lambda_2$.
\end{defi}

Divisible design digraphs are natural generalization of doubly 
regular asymmetric digraphs. Note that the adjacency matrix of a DDD with
$m=1$, $n=1$, or $\lambda_1 = \lambda_2$ is the incidence matrix of a symmetric design. In this case we call the DDD improper, otherwise it is proper.

An incidence structure with $v$ points and the constant block size $k$ is a 
(group) divisible design with parameters $(v,k, \lambda_1, \lambda_2, m,n)$ 
whenever the point set can be partitioned into $m$ classes of size $n$, 
such that two points from the same class are incident with exactly $\lambda_1$ common blocks, and two points from different classes are 
incident with exactly $\lambda_2$ common blocks. A divisible design $D$ is said to be symmetric (or to have the dual property) if the dual of $D$ is a 
divisible design with the same parameters as $D$. The definition of
a DDD yields the following theorem.

\begin{thm}
If $\Gamma$ is a divisible design digraph with parameters 
$(v,k, \lambda_1, \lambda_2, m,n)$ then its 
neighbourhood design is a symmetric divisible design 
$(v,k, \lambda_1, \lambda_2, m,n)$.
\end{thm} 

We say that a $( 0,1 )$-matrix $X$ is skew if $X+X^t$ is a $( 0,1 )$-matrix.
Thus the adjacency matrix of a DDD is skew.
If $D$ is a symmetric divisible design $(v,k, \lambda_1, \lambda_2, m,n)$ 
that has a skew incidence matrix, then $D$ is the neighbourhood design of a 
divisible design digraph with parameters $(v,k, \lambda_1, \lambda_2, m,n)$.

\bigskip

Let $G$ be a group and $S$  a subset of $G$ not containing the identity element of the group, which will be denoted by $e$. The vertices of the Cayley digraph $\Cay(G,S)$ are the elements 
of the group $G$, and its arcs are all the couples $(g,gs)$ with $g \in G$ and $s \in S$. Goryainov, Kabanov, and Shalaginov \cite{kabanov} studied divisible design
Cayley graphs. In this paper we study divisible design Cayley digraphs. We present some constructions for divisible design Cayley digraphs and give a classification of such digraphs
on $v \le 27$ vertices. The classification leads to the proof of existence of some divisible design digraphs with certain parameters which were previously undecided (see \cite{ddd}).

\section{Divisible design Cayley digraphs}\label{CDDDs}

The following well-known theorem (see \cite{sabidussi, ming-yao}) provides a characterization of Cayley digraphs.

\begin{thm}\label{regular}
A digraph $D=(V,A)$ is a Cayley digraph of a group if and only 
if $Aut(D)$ contains a regular subgroup. 
\end{thm}

The following characterization of divisible design Cayley graphs is given in \cite{kabanov}.
Note that a Deza graph with parameters $(n,k,b,a)$ is a $k$-regular graph with $n$ vertices in which any two vertices have $a$ or $b$ $(a \le b)$ common neighbours. 

\begin{thm} \label{cosets-DDG}
Let $\Cay(G,S)$ be a Deza graph with parameters $(v,k,b,a)$ and $SS^{-1}=aA+bB+k{e}$, where $A$, $B$ and $\{e\}$ be a partition of G. If either $A \cup \{e\}$ or $B \cup \{e\}$ is a subgroup of $G$, 
then $\Cay(G,S)$ is a DDG and the right cosets of this subgroup give a canonical partition of this graph.  
Conversely, if $\Cay(G,S)$ is a DDG, then the class of its canonical partition which contains the identity of G is a subgroup of G and classes of the canonical partition of DDG
coincides with the cosets of this subgroup.
\end{thm}

Using similar arguments one can prove the following theorem.

\begin{thm} \label{cosets-DDD}
Let $\Cay(G,S)$ be a digraph, $|S|=k$, 
$S \cap S^{-1} = \emptyset$, and $SS^{-1}=aA+bB+k{e}$, where $A$, $B$ and $\{e\}$ be a partition of G. If either $A \cup \{e\}$ or $B \cup \{e\}$ is a subgroup of $G$, 
then $\Cay(G,S)$ is a DDD and the right cosets of this subgroup give a canonical partition of this digraph.  
Conversely, if $\Cay(G,S)$ is a DDD, then the class of its canonical partition which contains the identity of G is a subgroup of G and classes of the canonical partition of DDD
coincides with the cosets of this subgroup.
\end{thm}

We omit a proof of Theorem \ref{cosets-DDD} since it is similar to the proof of Theorem \ref{cosets-DDG} given in \cite{kabanov}. 
The condition $S \cap S^{-1} = \emptyset$ ensures that the adjacency matrix of $\Cay(G,S)$ is skew. We found Theorem \ref{cosets-DDD} very useful when constructing divisible design Cayley digraphs
with the help of a computer. 

\bigskip

In the sequel we give constructions of divisible design Cayley digraphs
and nonexistence results. 
Throughout the paper we denote by $I_v$, $O_v$ and $J_v$ the identity matrix, the zero-matrix and the all-one matrix of size $v \times v$, respectively.

\subsection{Nonexistence results}\label{nonexistence}

A list of feasible parameters for DDDs on at most 27 vertices is given in \cite{ddd}. In this section we establish nonexistence of some 
divisible design Cayley digraphs with feasible parameters.
The results are obtained using Magma \cite{magma} and GAP \cite{gap-digraphs}.

\begin{thm} \label{thm-nonex}
The following divisible design Cayley digraphs do not exist: 
\begin{displaymath}   
\begin{tabular}{l l l l}
$(12,4,2,1,6,2)$ \hspace{0.5cm}  & $(20,5,2,1,10,2)$ \hspace{0.5cm}  & $(24,10,3,4,8,3)$ \hspace{0.5cm}  & $(24,10,6,3,3,8)$ \\
$(16,7,2,3,4,4)$ \hspace{0.5cm}  & $(20,9,3,4,4,5)$  \hspace{0.5cm}  & $(24,11,10,4,6,4)$ \hspace{0.5cm} & $(27,12,6,5,9,3)$ \\
$(18,6,0,2,6,3)$ \hspace{0.5cm}  & $(22,5,0,1,11,2)$ \hspace{0.5cm}  & $(24,9,4,3,6,4)$ \hspace{0.5cm}   & $(27,11,7,4,9,3)$ \\
$(20,8,2,3,10,2)$ \hspace{0.5cm} & $(24,10,2,4,12,2)$ \hspace{0.5cm} & $(24,11,4,5,4,6)$ \hspace{0.5cm}  & $(27,8,4,2,9,3)$  \\
$(20,7,6,2,10,2)$ \hspace{0.5cm} & $(24,9,6,3,12,2)$ \hspace{0.5cm}  &  &  \\
\end{tabular}                
\end{displaymath}
\end{thm}

\subsection{Constructions of divisible design Cayley digraphs from Paley designs} \label{constructions-Paley}

Let $q$ be a prime power. If $q \equiv 3\ (mod\ 4)$ then the set of non-zero squares in $\GF(q)$ forms a difference set in the additive group of $\GF(q)$, and in case $q \equiv 1\ (mod\ 4)$ 
the set of non-zero squares in $\GF(q)$ forms a partial difference set in the additive group of $\GF(q)$. 
The conclusion is  that the case $q \equiv 3\ (mod\ 4)$ yields a Cayley digraph, and the case $q \equiv 1\ (mod\ 4)$ yields a Cayley graph. 
The adjacency matrix of the Cayley digraph obtained in the case $q \equiv 3\ (mod\ 4)$ is the incidence matrix of a symmetric design called a Paley design, which is a Hadamard design with parameters 
$(q,\frac{q-1}{2},\frac{q-3}{4})$. The fact that $-1$ is not a square in the field $\GF(q)$ when $q \equiv 3\ (mod\ 4)$ implies that the incidence matrix of a Paley design is skew. The graph
obtained in the case $q \equiv 1\ (mod\ 4)$ is called a Paley graph, which is a strongly regular graph with parameters $(q,\frac{q-1}{2},\frac{q-5}{4},\frac{q-1}{4})$.

\begin{thm}\label{diag-gen}
Let $A$ be the incidence matrix of a symmetric $(v,k, \lambda)$ design with a regular automorphism group $G$. If $A$ is skew, 
then there exists a divisible design Cayley digraph with parameters $(vt,k,\lambda,0,t,v)$.
\end{thm}
\begin{proof}
By \cite[Lemma 4.1]{ddd} the Kronecker product $I_t \otimes A$ is the adjacency matrix of a DDD with parameters $(vt,k,\lambda,0,t,v)$.
It is obvious that the direct product of the group $G$ and the cyclic group $Z_t$ acts regularly on this DDD. By Theorem \ref{regular} the constructed DDD is a Cayley digraph.
\end{proof}

\begin{cor}\label{diag}
Let $v$ be a prime power, $v \equiv 3\ (mod\ 4)$, and let $t$ be a non-negative integer. Then there exists a divisible design Cayley digraph with parameters $(vt,\frac{v-1}{2},\frac{v-3}{4},0,t,v)$.
\end{cor}
\begin{proof}
Let $D$ be the incidence matrix of the Paley design with parameters
$(v,\frac{v-1}{2},\frac{v-3}{4})$. By the Paley construction $D$ is a skew matrix. By Theorem \ref{diag-gen} the Kronecker product $I_t \otimes D$ is the adjacency matrix of a divisible design Cayley
digraph with parameters $(vt,\frac{v-1}{2},\frac{v-3}{4},0,t,v)$. The direct product of the additive group of $\GF(v)$ and the cyclic group $Z_t$ acts regularly on this digraph. 
\end{proof}

\begin{thm}\label{all-one-gen}
Let $A$ be the incidence matrix of a symmetric $(v,k, \lambda)$ design with a regular automorphism group $G$. If $A$ is skew, 
then there exists a divisible design Cayley digraph with parameters $(vn,kn,kn,\lambda n,v,n)$.
\end{thm}
\begin{proof}
By \cite[Lemma 4.2]{ddd} the Kronecker product $A \otimes J_n$ is the adjacency matrix of a DDD with parameters $(vn,kn,kn,\lambda n,v,n)$.
This DDD admits a regular action of the direct product of the cyclic group $Z_n$ and the group of $G$, so it is a divisible design Cayley digraph.
\end{proof}

\begin{cor}\label{all-one}
Let $v$ be a prime power, $v \equiv 3\ (mod\ 4)$, and $n$ be a non-negative integer. There exists a divisible design Cayley digraph with parameters 
$(vn,\frac{v-1}{2}n,\frac{v-1}{2}n,\frac{v-3}{4} n,v,n)$.
\end{cor}
\begin{proof}
Let $D$ be the incidence matrix of the Paley design with parameters
$(v,\frac{v-1}{2},\frac{v-3}{4})$. The matrix $D$ is skew.
By Theorem \ref{all-one} the Kronecker product $D \otimes J_n$ is the adjacency matrix of a divisible design Cayley digraph with parameters $(vn,\frac{v-1}{2}n,\frac{v-1}{2}n,\frac{v-3}{4} n,v,n)$, 
which admits a regular action of the direct product of the cyclic group $Z_n$ and the additive group of $\GF(v)$.
\end{proof}

\begin{thm}\label{square}
Let $v$ be a prime power, $v \equiv 3\ (mod\ 4)$. Then there exists a 
divisible design Cayley digraph with parameters 
$(v^2,v\frac{v-1}{2},v\frac{v-3}{4},(\frac{v-1}{2})^2,v,v)$.
\end{thm}
\begin{proof}
Let $D_1$ be the incidence matrix of the Paley design with parameters
$(v,\frac{v-1}{2},\frac{v-3}{4})$, and $\overline{D_1}$ be the incidence matrix of its complementary design. 
By the Paley construction $D_1$ is a skew matrix.
Let $D=D_1R$, where $R$ is the back diagonal identity matrix, then $D$ is a symmetric matrix and an incidence matrix of a Hadamard 
$(v, \frac{v-1}{2}, \frac{v-3}{4})$ design.
Replace each entry value 1 of the matrix $D$ by $D_1$, and each entry value 0 of $D$ by $\overline{D}_1 - I_v$. According to \cite[Theorem 4.1]{ddd},
the resulting matrix $M$ is the adjacency matrix of a 
$DDD(v^2,v\frac{v-1}{2},v\frac{v-3}{4},(\frac{v-1}{2})^2,v,v)$.
The construction shows that this DDD admits a regular action of the direct product of the additive group of $\GF(v)$ by itself.
\end{proof}

\begin{thm}\label{3-5}
Let $t$ be a non-negative integer. If $4t+3$ and $4t+5$ are prime powers,
then there exists a divisible design Cayley digraph with parameters
$((4t+5)(4t+3),(4t+4)(2t+1),l(4t+4),(2t+1)^2,4t+5,4t+3)$. 
\end{thm}
\begin{proof}
Let $D$ be the incidence matrix of the Paley design with parameters 
$(4t+3,2t+1,t)$ and $C$ be the adjacency matrix of the Paley graph on $4t+1$ vertices, which is a strongly regular graph with parameters $(4t+1,2t,t-1,t)$. $D$ is a skew matrix, 
and $C$ is a symmetric matrix. Let $\overline{D}=J_{4t+3}-D$ and $\overline{C}=J_{4t+1}-C$.
According to \cite[Theorem 4.2]{ddd}, the matrix $M=C \otimes D + (\overline C - I_{4t+1}) \otimes (\overline D - I_{4t+3})$ 
is the adjacency matrix of a $DDD((4t+5)(4t+3),(4t+4)(2t+1),t(4t+4),(2t+1)^2,4t+5,4t+3)$.
It is clear from the construction of the matrix $M$ that this DDD admits a regular action of the direct product of the additive group of $\GF(4t+3)$ and the additive group of $\GF(4t+1)$.
\end{proof}

\begin{thm}\label{Fano}
Let $t$ be a non-negative integer such that $4t+3$ is a prime power. Then there exists a divisible design Cayley digraph with parameters $(28t+21,8t+7,4t+3,2t+2,7,4t+3)$.
\end{thm}
\begin{proof}
Let $D$ be the incidence matrix of a Paley design with parameters $(7,3,1)$, and let $D_1$ be the incidence matrix of a Paley design with parameters $4t+3,2t+1,t$. 
Further, let $\overline D_1=J_{4t+3}-D_1$. By \cite[Lemma 4.4]{ddd}, the matrix $D \otimes \overline{D}_1 + I_7 \otimes D_1$ is the adjacency matrix of a $DDD(28t+21,8t+7,4t+3,2t+2,7,4t+3)$. 
It is clear from the construction of this DDD that it admits a regular action of the direct product of the additive group of $\GF(4t+3)$ and the additive group of $\GF(7)$.
\end{proof}

\begin{thm}\label{2Paley}
Let $t_1$ and $t_2$ be non-negative integers such that $4t_1+3$ and $4t_2+3$  
are prime powers. Then there exists a divisible design Cayley digraph with parameters $((4t_1+3)(4t_2+3),(2t_1+1)(4t_2+3)+2t_2+1,(2t_1+1)(4t_2+3) + t_2,(4t_2+3) t_1 + 2t_2+1,4t_1+3,4t_2+3)$.
\end{thm}
\begin{proof}
Let $D_1$ be the incidence matrix of a Paley design with parameters 
$(4t_1+3,2t_1+1,t_1)$ and $D_2$ be the incidence matrix of a Paley design with parameters $(4t_2+3,2t_2+1,t_2)$. Replace each diagonal entry of the matrix $D_1$ by $D_2$, 
each off-diagonal entry value 0 of $D_1$ by $O_{4t_2+3}$ and each entry value 1 of $D_1$ by $J_{4t_2+3}$.
By \cite[Lemma 4.5]{ddd}, the resulting matrix is the adjacency matrix of a $DDD((4t_1+3)(4t_2+3),(2t_1+1)(4t_2+3)+2t_2+1,(2t_1+1)(4t_2+3) + t_2,(4t_2+3) t_1 + 2t_2+1,4t_1+3,4t_2+3)$. 
The direct product of the additive group of $\GF(4t_1+3)$ and the additive group of $\GF(4t_2+3)$ acts regularly on the set of vertices of that DDD.
\end{proof}

\begin{thm}\label{Paley-circ}
Let $v$ be a prime power, $v \equiv 3\ (mod\ 4)$, and $n$ be an integer, $n \ge 3$. Then there exists a divisible design Cayley digraph with parameters 
$(vn,n \frac{v-1}{2} +1,n \frac{v-1}{2}, n \frac{v-3}{4} + 1, v,n)$.
\end{thm}
\begin{proof}
Let $D$ be the incidence matrix of a Paley design with parameters $(v,\frac{v-1}{2},\frac{v-3}{4})$ and $C$ be a $n \times n$ circulant matrix $C=circ(0,1,0,\ldots,0)$.
Replace each diagonal entry of the matrix $D$ by $C$, each off-diagonal entry value 0 of $D$ by $O_n$, and each entry value 1 of $D$ by $J_n$.
The resulting matrix is the adjacency matrix of a $DDD(vn,n \frac{v-1}{2} +1,n \frac{v-1}{2}, n \frac{v-3}{4} + 1, v,n)$.
The direct product of the additive group of $\GF(v)$ and the cyclic group $Z_n$ acts regularly on the set of vertices of the constructed DDD.
\end{proof}

\subsection{Other constructions of divisible design Cayley digraphs} \label{constructions-other}

The following theorem gives a construction of a divisible design Cayley digraph from the quaternion group $Q_8$.

\begin{thm}\label{v=8}
There exists a divisible design Cayley digraph with parameters $(8,3,0,1,4,2)$.
\end{thm}
\begin{proof}
Let $G= \langle \bar{e}, i, j, k \ | \ \bar{e}^2= e, i^2=j^2=k^2=ijk=\bar{e}^2 \rangle$ be the quaternion group $Q_8$ and $S= \{i,j,k \}$. Then $S \cap S^{-1}= \emptyset$ and
$SS^{-1}=A+0B+3e$, where $A= \{i, i^{-1}, j, j^{-1}, k, k^{-1} \}$ and $B= \{\bar{e}\}$. Obviously, $B \cup \{ e \}$ is a subgroup of $G$, so $\Cay(G,S)$ is a DDD and 
the right cosets of $B$ give a canonical partition of a divisible design Cayley digraph with parameters $(8,3,0,1,4,2)$.
\end{proof}

\begin{thm}\label{cyclic}
For every odd integer $n$, $n \ge 3$, there exists a divisible design Cayley digraph with parameters $(4n,n+2,n-2,2,4,n)$.
\end{thm}
\begin{proof}
Let $C_n=circ(0,1,0, \ldots , 0)$ be a $n \times n$ matrix and let $\overline I_n= J_n-I_n$. Then
\begin{displaymath}
  \left[
\begin{tabular}{cccc}
$C_n$             &  $\overline I_n$   & $C_n$             & $I$              \\
$I$               & $C_n$              & $\overline I_n$   & $C_n$            \\
$C_n$             & $I$                & $C_n$             & $\overline I_n$  \\
$\overline I_n$   & $C_n$              & $I$               & $C_n$            \\
\end{tabular}                 \right]
\end{displaymath}
is the adjacency matrix of a $DDD$ with parameters $(4n,n+2,n-2,2,4,n)$ (see \cite[Theorem 4.6]{ddd}).
The direct product of the cyclic groups $Z_n$ and $Z_4$  acts regularly on that DDD.
\end{proof}

In the following theorem we give parameters of divisible design Cayley digraphs for $v \le 27$, $0 < \lambda_2 < k$ and $\lambda_1 < k$, whose existence were established using a computer.

\begin{thm}\label{comp}
There exist divisible design Cayley digraphs with the following parameters:
\begin{displaymath}   
\begin{tabular}{l l l l}
$(16,7,0,3,8,2)$ \hspace{0.5cm}  & $(24,7,0,2,8,3)$ \hspace{0.5cm}  & $(24,6,2,1,3,8)$ \hspace{0.5cm}  & $(26,9,0,3,13,2)$ \\
$(16,4,0,1,4,4)$ \hspace{0.5cm}  & $(24,5,0,1,6,4)$  \hspace{0.5cm}  & $(25,5,0,1,5,5)$ \hspace{0.5cm} & $(27,9,0,3,9,3)$ \\
$(24,11,0,5,12,2)$ \hspace{0.5cm} & $(24,8,4,2,4,6)$ &  &  \\
\end{tabular}                
\end{displaymath}
\end{thm}

\bigskip

A balanced generalized weighing matrix BGW$(v,k, \lambda)$ over a multiplicative group $G$ is a $v \times v$ matrix $W=[g_{ij}]$ with entries from $\overline{G}=G \cup \{ 0 \}$
such that each row of $W$ contains exactly $k$ nonzero entries, and for every $a,b \in \{ 1,2, \ldots ,v \}$, $a \neq b$, the multi-set 
$\{ g_{aj} g_{bj}^{-1} | \ 1 \le j \le v, \ g_{aj} \neq 0, \ g_{bj} \neq 0 \}$ contains exactly $\lambda / |G|$ copies of each element of $G$. 

In this work we concentrate only on balanced generalized weighing matrices BGW$(n+1,n,n-1)$ over a cyclic group $C_{n-1}$ of order $n-1$, for $n$ a prime power. Such a balanced generalized weighing matrix $W$ has zero diagonal entries and if the entries of $W$ are assumed to belong to the finite field GF$(n)$, then $W^t=-W$, and thus we call it to be {\it skew}, see \cite{ik}. 

\begin{thm}
 There is a divisible design  digraph with parameters $$(n^2-1,n,0,1,n+1,n-1),$$ for each $n$ an odd prime power. 
\end{thm}
\begin{proof}
 Let $W=[w_{ij}]$ be a skew BGW$(n+1,n,n-1)$ over the cyclic group $C_{n-1}$ generated by the circulant matrix $U=circ(0,1,0,\cdots,0)$ of order $n-1$. Then the matrix  $[w_{ij}R_{n-1}]$, where $R_{n-1}$ is the back diagonal identity matrix of order $n-1$ is the desired digraph.
\end{proof}

In all cases that we have checked by computer  
these BGWs produce divisible design Cayley digraphs with parameters $(n^2-1,n,0,1,n+1,n-1)$. That leads us to the following conjecture.

\begin{con}
Let $n$ be an odd prime power. Then there exists a divisible design Cayley digraph with parameters $(n^2-1,n,0,1,n+1,n-1)$.
\end{con}

\section{Small parameters} \label{small}

All putative parameter sets $(v, k, \lambda_1, \lambda_2, m, n)$ for DDDs on at most 27 vertices are given in \cite{ddd}. 
In {\rm T}ables \ref{parameters1}, \ref{parameters2} and \ref{parameters3} we give the information of the existence of DDDs on at most 27 vertices and the existence of the divisible design Cayley digraphs.
We also give the number of divisible design Cayley digraphs with certain parameters, up to isomorphism.
The cases when $\lambda_1=k$ or $\lambda_2=0$ are omitted.
Examples of divisible design Cayley digraphs with $\lambda_1=k$ or $\lambda_2=0$ are given in Theorem \ref{diag-gen} and Theorem \ref{all-one-gen}.

\begin{rem}
The existence of DDDs with parameters $(24,7,0,2,8,3)$, $(24,8,4,2,4,6)$, $(24,6,2,1,3.8)$, $(26,9,0,3,13,2)$ and $(27,9,0,3,9,3)$ were all open problems in (see \cite{ddd}).
Hence, by constructing divisible design Cayley digraphs with parameters $(24,7,0,2,8,3)$, $(24,8,4,2,4,6)$, $(24,6,2,1,3,8)$, $(26,9,0,3,13,2)$ and $(27,9,0,3,9,3)$ we proved the existence of DDDs
with these parameters.
\end{rem}

\begin{table}[H] {\scriptsize 
\caption{\footnotesize Feasible parameters of proper $DDDs$ with $v \le 20$, $0 < \lambda_2 < k$, $\lambda_1 < k$}
\label{parameters1}
\begin{center}
\begin{tabular}{r r r r r r | l l | l l r}
$v$ & $k$  & $\lambda_1$ & $\lambda_2$ & $m$ & $n$ & existence & reference 
& Cayley & reference & $\#$ DDCD
\\
\hline\hline
8 & 3 & 0 & 1 & 4 & 2 & yes & \cite{ddd} & yes & Theorem \ref{v=8} & 1 \\
9 & 4 & 3 & 1 & 3 & 3 & yes & \cite{ddd} & yes & Theorem \ref{2Paley}, Theorem \ref{Paley-circ} & 1\\
9 & 4 & 0 & 2 & 3 & 3 & no  & \cite{ddd} & no  &  & 0\\
9 & 3 & 0 & 1 & 3 & 3 & yes & \cite{ddd} & yes & Theorem \ref{square} & 1\\
12 & 5 & 1 & 2 & 4 & 3 & yes & \cite{ddd} & yes & Theorem \ref{cyclic} & 2 \\
12 &  5 &  4 & 1 &  3 &  4 & yes & \cite{ddd} & yes & Theorem \ref{Paley-circ} & 1\\
12 &  5 &  0 & 2 &  6 &  2 & no  & \cite{ddd} & no & & 0\\
12 &  4 &  2 & 1 &  6 &  2 & yes & \cite{ddd} & no & Theorem \ref{thm-nonex} & 0\\
12 &  4 &  0 & 2 &  2 &  6 & no  & \cite{ddd} & no &  & 0\\
12 &  3 &  0 & 1 &  2 &  6 & no  & \cite{ddd} & no &  & 0\\
14 &  4 &  0 & 1 &  7 &  2 & no  & \cite{ddd} & no & & 0\\
14 &  5 &  1 & 2 &  2 &  7 & no  & \cite{ddd} & no & & 0\\
15 &  6 &  3 & 2 &  5 &  3 & no  & \cite{ddd} & no & & 0\\
15 &  5 &  4 & 1 &  5 &  3 & no  & \cite{ddd} & no & & 0\\
15 &  4 &  0 & 1 &  5 &  3 & yes & \cite{ddd} & yes & Theorem \ref{3-5} & 10\\
15 &  6 &  0 & 3 &  3 &  5 & no  & \cite{ddd} & no &  & 0\\
15 &  6 &  5 & 1 &  3 &  5 & yes & \cite{ddd} & yes & Theorem \ref{Paley-circ} & 1\\
15 &  5 &  0 & 2 &  3 &  5 & no  & \cite{ddd} & no &  &  0\\
16 &  7 &  0 & 3 &  8 &  2 & yes & \cite{ddd} & yes & Theorem \ref{comp} & 1\\
16 &  7 &  2 & 3 &  4 &  4 & yes & \cite{ddd} & no & Theorem \ref{thm-nonex} & 0\\
16 &  7 &  6 & 2 &  4 &  4 & no  & \cite{ddd} & no &  & 0\\
16 &  4 &  0 & 1 &  4 &  4 & yes & \cite{ddd} & yes & Theorem \ref{comp} & 6 \\
18 &  5 &  4 & 1 &  9 &  2 & no  & \cite{ddd} & no &  & 0\\
18 &  7 &  6 & 2 &  6 &  3 & no  & \cite{ddd} & no & & 0\\
18 &  6 &  0 & 2 &  6 &  3 & yes & \cite{ddd} & no & Theorem \ref{thm-nonex} & 0\\
18 &  8 &  4 & 3 &  3 &  6 & no  & \cite{ddd} & no & & 0\\
18 &  7 &  6 & 1 &  3 &  6 & yes & \cite{ddd} & yes & Theorem \ref{Paley-circ} & 2\\
18 &  4 &  0 & 1 &  3 &  6 & no  & \cite{ddd} & no & &  0\\
18 &  7 &  3 & 2 &  2 &  9 & no  & \cite{ddd} & no & & 0\\
20 &  9 &  0 & 4 & 10 &  2 & no  & \cite{ddd} & no & & 0\\
20 &  8 &  2 & 3 & 10 &  2 &  ?  &        -   & no& Theorem \ref{thm-nonex} & 0 \\
20 &  7 &  6 & 2 & 10 &  2 &  ?  &        -   & no & Theorem \ref{thm-nonex} &     0\\
20 &  5 &  2 & 1 & 10 &  2 &  ?  &        -   & no & Theorem \ref{thm-nonex} &     0\\
20 &  9 &  8 & 3 &  5 &  4 & no  & \cite{ddd} & no & & 0\\
20 &  9 &  3 & 4 &  4 &  5 &  ?  &        -   & no & Theorem \ref{thm-nonex} &    0\\
20 & 7 & 3 & 2 & 4 & 5 & yes & \cite{ddd} & yes & Theorem \ref{cyclic} & 4\\
20 &  6 &  0 & 2 &  4 &  5 & no  & \cite{ddd}  & no & & 0\\
20 &  8 &  4 & 2 &  2 & 10 & no  & \cite{ddd}  & no & &0\\
20 &  6 &  0 & 3 &  2 & 10 & no  & \cite{ddd}  & no & & 0\\
20 &  5 &  0 & 2 &  2 & 10 & no  & \cite{ddd}  & no & & 0\\
\hline\hline
\end{tabular} \end{center} }
\end{table}

\begin{table}[H] {\scriptsize
\caption{\footnotesize Feasible parameters of proper $DDDs$ with $21 \le v \le 25$, $0 < \lambda_2 < k$, $\lambda_1 < k$}
\label{parameters2}
\begin{center}
\begin{tabular}{r r r r r r | l l | l l r}
$v$ & $k$  & $\lambda_1$ & $\lambda_2$ & $m$ & $n$ & existence & reference 
& Cayley & reference & $\#$ DDCD
\\
\hline\hline
21 & 10 &  0 & 5 &  7 &  3 & no  & \cite{ddd} & no &  & 0\\
21 & 10 & 9 & 4 & 7 & 3 & yes & \cite{ddd} & yes & Theorem \ref{2Paley}, Theorem \ref{Paley-circ} & 1\\
21 &  9 &  0 & 4 &  7 &  3 & no  & \cite{ddd} & no &  & 0\\
21 &  8 &  1 & 3 &  7 &  3 &  ?  &        -   &  no&  & 0\\
21 &  7 &  3 & 2 &  7 &  3 & yes & \cite{ddd} & yes & Theorem \ref{Fano} & 1 \\
21 & 10 &  1 & 6 &  3 &  7 & no  & \cite{ddd} & no &  & 0\\
21 & 10 & 8 & 3 & 3 & 7 & yes & \cite{ddd} & yes & Theorem \ref{2Paley} & 1\\
21 &  9 &  5 & 3 &  3 &  7 & no  & \cite{ddd} & no & & 0\\
21 &  8 &  0 & 4 &  3 &  7 & no  & \cite{ddd} & no & & 0\\
21 &  8 &  7 & 1 &  3 &  7 & yes & \cite{ddd} & yes & Theorem \ref{Paley-circ} & 1 \\
21 &  7 &  0 & 3 &  3 &  7 & no  & \cite{ddd} & no & & 0\\
22 &  5 &  0 & 1 & 11 &  2 &  ?  &        -   & no & Theorem \ref{thm-nonex} &  0\\
22 &  9 &  5 & 2 &  2 & 11 & no  & \cite{ddd} & no & & 0\\
24 & 11 &  0 & 5 & 12 &  2 & yes & \cite{ddd} & yes & Theorem \ref{comp} & 1 \\
24 & 10 &  2 & 4 & 12 &  2 & yes & \cite{ddd} & no & Theorem \ref{thm-nonex} & 0\\
24 &  9 &  6 & 3 & 12 &  2 &  ?  &        -   & no & Theorem \ref{thm-nonex} &   0\\
24 & 10 &  3 & 4 &  8 &  3 &  ?  &        -   & no & Theorem \ref{thm-nonex} &    0\\
24 &  8 &  7 & 2 &  8 &  3 & no  & \cite{ddd} & no & & 0\\
24 &  7 &  0 & 2 &  8 &  3 & yes & Theorem \ref{comp} & yes & Theorem \ref{comp} & 6\\
24 & 11 & 10 & 4 &  6 &  4 &  ?  &        -   & no & Theorem \ref{thm-nonex} &  0\\
24 &  9 &  4 & 3 &  6 &  4 &  ?  &        -   & no & Theorem \ref{thm-nonex} &  0\\
24 &  5 &  0 & 1 &  6 &  4 & yes & \cite{ddd} & yes & Theorem \ref{comp} & 17\\
24 & 11 &  4 & 5 &  4 &  6 &  ?  &        -   & no & Theorem \ref{thm-nonex} &  0\\
24 & 10 &  0 & 5 &  4 &  6 & no  & \cite{ddd} & no &  & 0\\
24 &  9 &  0 & 4 &  4 &  6 & no  & \cite{ddd} & no &  & 0\\
24 &  8 &  4 & 2 &  4 &  6 & yes & Theorem \ref{comp} & yes & Theorem \ref{comp} & 8\\
24 & 11 &  2 & 6 &  3 &  8 & no  & \cite{ddd} & no &  & 0\\
24 & 10 &  6 & 3 &  3 &  8 &  ?  &        -   & no & Theorem \ref{thm-nonex} &  0\\
24 &  9 &  8 & 1 &  3 &  8 & yes & \cite{ddd} & yes & Theorem \ref{Paley-circ} & 2\\
24 &  6 &  2 & 1 &  3 &  8 & yes & Theorem \ref{comp} & yes & Theorem \ref{comp} &  14 \\
24 & 10 &  6 & 2 &  2 & 12 & no  & \cite{ddd} & no & & 0\\
24 &  9 &  0 & 6 &  2 & 12 & no  & \cite{ddd} & no &  & 0\\
24 &  8 &  4 & 1 &  2 & 12 & no  & \cite{ddd} & no &  &0\\
24 &  4 &  0 & 1 &  2 & 12 & no  & \cite{ddd} & no &  & 0\\
25 & 12 &  3 & 6 &  5 &  5 & no  & \cite{ddd} & no &  & 0\\
25 & 12 &  8 & 5 &  5 &  5 & no  & \cite{ddd} & no &  & 0\\
25 &  9 &  8 & 2 &  5 &  5 & no  & \cite{ddd} & no &  & 0\\
25 &  8 &  4 & 2 &  5 &  5 & no  & \cite{ddd} & no &  & 0\\
25 &  5 &  0 & 1 &  5 &  5 & yes & \cite{ddd} & yes & Theorem \ref{comp} & 1\\
\hline\hline
\end{tabular}\end{center}  }
\end{table}

\begin{table}[H] {\scriptsize
\caption{\footnotesize Feasible parameters of proper $DDDs$ with $26 \le v \le 27$, $0 < \lambda_2 < k$, $\lambda_1 < k$}
\label{parameters3}
\begin{center}
\begin{tabular}{r r r r r r | l l | l l r}
$v$ & $k$  & $\lambda_1$ & $\lambda_2$ & $m$ & $n$ & existence & reference 
& Cayley & reference & $\#$ DDCD
\\
\hline\hline
26 &  9 &  0 & 3 & 13 &  2 & yes & Theorem \ref{comp} & yes & Theorem \ref{comp} & 2 \\
26 & 11 &  7 & 2 &  2 & 13 & no  & \cite{ddd} & no & & 0\\
26 & 10 &  1 & 6 &  2 & 13 & no  & \cite{ddd} & no & & 0\\
27 & 12 &  6 & 5 &  9 &  3 &  ?  &        -   & no & Theorem \ref{thm-nonex} &  0\\
27 & 11 &  7 & 4 &  9 &  3 &  ?  &        -   & no & Theorem \ref{thm-nonex}  &   0 \\
27 & 10 &  9 & 3 &  9 &  3 & no  & \cite{ddd} & no & & 0\\
27 &  9 &  0 & 3 &  9 &  3 & yes & Theorem \ref{comp} & yes & Theorem \ref{comp} & 3\\
27 &  8 &  4 & 2 &  9 &  3 &  ?  &        -   & no & Theorem \ref{thm-nonex} &   0\\
27 &  6 &  3 & 1 &  9 &  3 & no  & \cite{ddd} & no &  &0\\
27 & 12 &  3 & 6 &  3 &  9 & no  & \cite{ddd} & no &  & 0\\
27 & 11 &  7 & 3 &  3 &  9 & no  & \cite{ddd} & no & & 0\\
27 & 10 &  0 & 5 &  3 &  9 & no  & \cite{ddd} & no &  & 0\\
27 & 10 &  9 & 1 &  3 &  9 & yes & \cite{ddd} & yes & Theorem \ref{Paley-circ} & 2\\
27 &  9 &  0 & 4 &  3 &  9 & no  & \cite{ddd} & no &  & 0\\
27 &  7 &  3 & 1 &  3 &  9 & no  & \cite{ddd} & no & & 0\\
\hline\hline
\end{tabular} \end{center} }
\end{table}

In Table \ref{DDCDs} we give parameters of divisible design Cayley digraphs $(v, k, \lambda_1, \lambda_2, m, n)$ with $v \le 27$, $0 < \lambda_2 < k$, $\lambda_1 < k$. 
Divisible design Cayley digraphs with parameters $(18,7,6,1,3,6)$,
$(21,10,9,4,7,3)$, $(21,7,3,2,7,3)$, $(21,10,8,3,3,7)$, $(21,8,7,1,3,7)$, $(24,11,0,5,12,2)$, $(24,9,8,1,3,8)$ and $(27,10,9,1,3,9)$ admit a regular action of more than one group, so each of them 
can be constructed as a Cayley graph using different groups.

\begin{table}[H] {\scriptsize
\caption{\footnotesize Proper divisible design Cayley digraphs with $v \le 27$, $0 < \lambda_2 < k$, $\lambda_1 < k$}
\label{DDCDs}
\begin{center}
\begin{tabular}{ c|c|c}
DDCD & regular groups acting on a DDCD  & $\#$ DDCD
\\
\hline\hline
(8,3,0,1,4,2) &  $Q_8$ & 1 \\
\hline
(9,4,3,1,3,3) & $E_9$ & 1 \\
(9,3,0,1,3,3) & $E_9$ & 1 \\
\hline
(12,5,1,2,4,3) & $Z_3:Z_4$ & 1 \\
							& $Z_{12}$ & 1  \\
(12,5,4,1,3,4) & $Z_{12}$ & 1 \\
\hline
(15,4,0,1,5,3) & $Z_{15}$ & 10 \\
(15,6,5,1,3,5) & $Z_{15}$ & 1 \\
\hline
(16,7,0,3,8,2) & $Q_{16}$ & 1 \\
(16,4,0,1,4,4) & $Q_{16}$ & 6 \\
\hline
(18,7,6,1,3,6) & $Z_{18}$ ($Z_6\times Z_3$, $Z_3\times S_3$) & 2  \\ 
\hline
(20,7,3,2,4,5) & $Z_5:Z_4$ & 1 \\
							& $Z_{20}$ & 3 \\
\hline
(21,10,9,4,7,3) & $Z_{21}$ ($Z_7:Z_3$) & 1 \\
(21,7,3,2,7,3) & $Z_{21}$ ($Z_7:Z_3$) & 1 \\
(21,10,8,3,3,7)  & $Z_{21}$ ($Z_7:Z_3$) & 1 \\
(21,8,7,1,3,7) & $Z_{21}$ ($Z_7:Z_3$) & 1 \\ 
\hline
(24,11,0,5,12,2) & $Z_3:Q_8$ ($SL(2,3)$) & 1  \\
(24,7,0,2,8,3) & $Z_{24}$ & 5 \\ 
							&  $Z_3 \times D_8$ & 1 \\
(24,5,0,1,6,4) & $Z_3: Z_8$ & 5 \\
						& $Z_{24}$ & 9 \\ 
						& $SL(2,3)$ & 3 \\
(24,8,4,2,4,6) & $Z_4\times S_3$ & 2 \\
						& $Z_2\times (Z_3:Z_4)$ & 3 \\
						& $Z_{12}\times Z_2$ & 3 \\
(24,9,8,1,3,8) & $Z_{24}$ ($SL(2,3)$, $Z_{12}\times Z_2$, $Z_3\times D_8$, $Z_3\times Q_8$)& 2  \\
(24,6,2,1,3,8) & $Z_{12}\times Z_2$ & 12 \\
							& $Z_3\times D_8$ & 2  \\
\hline
(25,5,0,1,5,5) &$E_{25}$ & 1  \\
\hline
(26,9,0,3,13,2) & $Z_{26}$ & 2   \\
\hline
(27,9,0,3,9,3) & $E_9:Z_3$ & 3  \\
(27,10,9,1,3,9) & $Z_{27}$ ($Z_9\times Z_3$, $E_9:Z_3$, $E_{27}$, $Z_9:Z_3$) & 2  \\
\hline\hline
\end{tabular} \end{center} }
\end{table}

\noindent {\bf Acknowledgement} \\
D. Crnkovi\' c and A. \v Svob were supported by {\rm C}roatian Science Foundation under the project 6732.
Hadi Kharaghani acknowledges the support of the Natural Sciences and Engineering Research Council of Canada (NSERC).
D. Crnkovi\' c thanks the University of Lethbridge for their hospitality.

\end{document}